\documentclass{article}
\usepackage{amsmath,amssymb,amsfonts,latexsym}
\usepackage{color,graphicx}
\usepackage[english]{babel}
\usepackage{amsthm}
\usepackage{authblk}
\usepackage{fancybox}
\usepackage{algorithm}

\usepackage{mathtools}
\usepackage{enumerate,xspace}
\usepackage{thm-restate}
\usepackage{siunitx}

\newtheorem{theorem}{Theorem}

\title{\large \bf \textsf{A Combinatorial Proof of the Existence of Dense subsets in $\mathbb{R}$ without the \lq\lq{Steinhaus}\rq\rq~like property.}}
\author{Arpan Sadhukhan }
\date{}
\affil{\textit{Eindhoven University of Technology}}
\affil{\textit{Department of Mathematics and Computer Science}}

\begin{document}

\maketitle

\noindent The Steinhaus theorem states that for any positive measurable set $S$, the difference set $S-S$ contains an open interval around the origin \cite{Tao}. 
It is somewhat natural to ask whether there exists a subset $P \subseteq \mathbb{R}$ such that neither $P-P$ nor $P^c-P^c$ contains any interval in $\mathbb{R}$. Clearly, by the Steinhaus theorem we know that if $P$ is Lebesgue measurable, then $P$ must have measure $0$, in which case we know $P^c-P^c$ has an open interval around $0$. Consequently, such sets must not be Lebesgue measurable. Also, in the domain of nonmeasurable subsets of $\mathbb{R}$, it is not entirely clear that such sets will exist. In this note we prove the existence of such sets which are also dense. The idea for the proof is similar to \cite{Thomas}. Somewhat related, but entirely different results, are found in \cite{Bartoszewicz}.

\begin{theorem}
There exists a dense subset $P \subseteq \mathbb{R}$ such that neither difference set $P-P$ nor $P^c-P^c$ contains any interval in $\mathbb{R}$.
\end{theorem}

\begin{proof}

Let $S=\{2N+\frac{2k+1}{3^n}:|2k+1|<3^n \: \text{and} \:  N, k \: \text{are integers}, n \in \mathbb{N}\}$. Thus $S$ is dense in $\mathbb{R}$.
Define an infinite graph $G$ on $\mathbb{R}$ such that two vertices $a,b \in \mathbb{R}$ are adjacent iff $|a-b| \in S$. Now if $G$ has an odd cycle $(v_1,v_2, \ldots v_{2n+1})$, then it is easy to see $0=|v_1-v_1|=2N'+\frac{2k'+1}{3^{n'}}$ where $N', k'$ are integers and $n'\in \mathbb{N}$, which is a contradiction. Hence, $G$ is bipartite.
Let $G=\bigcup_{\alpha \in J} G_\alpha$ where $G_\alpha, \alpha \in J$, represents the connected components of $G$ and $J$ is the index set. Using the axiom of choice, choose a vertex $v_\alpha \in G_\alpha \: \forall \alpha \in J$. Let $P_\alpha$ be the subset of vertices of $G_\alpha$ such that the shortest path from each element of $P_\alpha$ to $v_\alpha$ is of even length. Let $P=\bigcup_{\alpha \in J} P_\alpha$. Observe that the sets $P$ and $P^c$ are dense in $\mathbb{R}$. Also observe that no two vertices in $P$ and $P^c$ are adjacent as $G$ has no odd cycles.  Now, without loss of generality, assume $D=P-P$ contains an interval $I\subseteq \mathbb{R}$. Hence, $D \cap S$ is non-empty, which is a contradiction. This completes the proof.

\end{proof}

\bigskip
\footnoterule


\end{document}